\documentclass[11pt,oneside]{amsart}
\usepackage[latin9]{inputenc}
\usepackage{enumitem}
\usepackage{amsthm}
\usepackage{amssymb}

\makeatletter
\numberwithin{equation}{section}
\numberwithin{figure}{section}
\theoremstyle{plain}
\newtheorem{thm}{\protect\theoremname}[section]
\theoremstyle{plain}
\newtheorem{conjecture}[thm]{\protect\conjecturename}
\theoremstyle{remark}
\newtheorem{rem}[thm]{\protect\remarkname}
\ifx\proof\undefined
\newenvironment{proof}[1][\protect\proofname]{\par
\normalfont\topsep6\p@\@plus6\p@\relax
\trivlist
\itemindent\parindent
\item[\hskip\labelsep\scshape #1]\ignorespaces
}{%
\endtrivlist\@endpefalse
}
\providecommand{\proofname}{Proof}
\fi
\theoremstyle{plain}
\newtheorem{lem}[thm]{\protect\lemmaname}
\theoremstyle{plain}
\newtheorem{cor}[thm]{\protect\corollaryname}
\theoremstyle{remark}
\newtheorem{claim}[thm]{\protect\claimname}
 \newlist{casenv}{enumerate}{4}
 \setlist[casenv]{leftmargin=*,align=left,widest={iiii}}
 \setlist[casenv,1]{label={{\itshape\ \casename} \arabic*.},ref=\arabic*}
 \setlist[casenv,2]{label={{\itshape\ \casename} \roman*.},ref=\roman*}
 \setlist[casenv,3]{label={{\itshape\ \casename\ \alph*.}},ref=\alph*}
 \setlist[casenv,4]{label={{\itshape\ \casename} \arabic*.},ref=\arabic*}
\theoremstyle{plain}
\newtheorem{prop}[thm]{\protect\propositionname}

\usepackage{amsfonts}
\usepackage{nicefrac}
\usepackage{amscd}
\usepackage{a4wide}
\usepackage{url}

\makeatother

 \providecommand{\casename}{Case}
\providecommand{\claimname}{Claim}
\providecommand{\conjecturename}{Conjecture}
\providecommand{\corollaryname}{Corollary}
\providecommand{\lemmaname}{Lemma}
\providecommand{\propositionname}{Proposition}
\providecommand{\remarkname}{Remark}
\providecommand{\theoremname}{Theorem}

\begin{document}
\global\long\def\p{\mathbf{p}}
\global\long\def\q{\mathbf{q}}
\global\long\def\C{\mathfrak{C}}
\global\long\def\SS{\mathcal{P}}
 \global\long\def\image{\operatorname{im}}
\global\long\def\otp{\operatorname{otp}}
\global\long\def\dec{\operatorname{dec}}
\global\long\def\suc{\operatorname{suc}}
\global\long\def\pre{\operatorname{pre}}
\global\long\def\qe{\operatorname{qf}}
 \global\long\def\ind{\operatorname{ind}}
\global\long\def\Nind{\operatorname{Nind}}
\global\long\def\lev{\operatorname{lev}}
\global\long\def\Suc{\operatorname{Suc}}
\global\long\def\HNind{\operatorname{HNind}}
\global\long\def\minb{{\lim}}
\global\long\def\concat{\frown}
\global\long\def\cl{\operatorname{cl}}
\global\long\def\tp{\operatorname{tp}}
\global\long\def\id{\operatorname{id}}
\global\long\def\cons{\left(\star\right)}
\global\long\def\qf{\operatorname{qf}}
\global\long\def\ai{\operatorname{ai}}
\global\long\def\dtp{\operatorname{dtp}}
\global\long\def\acl{\operatorname{acl}}
\global\long\def\nb{\operatorname{nb}}
\global\long\def\limb{{\lim}}
\global\long\def\leftexp#1#2{{\vphantom{#2}}^{#1}{#2}}
\global\long\def\intr{\operatorname{interval}}
\global\long\def\atom{\emph{at}}
\global\long\def\I{\mathfrak{I}}
\global\long\def\uf{\operatorname{uf}}
\global\long\def\ded{\operatorname{ded}}
\global\long\def\Ded{\operatorname{Ded}}
\global\long\def\Df{\operatorname{Df}}
\global\long\def\Th{\operatorname{Th}}
\global\long\def\eq{\operatorname{eq}}
\global\long\def\Aut{\operatorname{Aut}}
\global\long\def\ac{ac}
\global\long\def\DfOne{\operatorname{df}_{\operatorname{iso}}}
\global\long\def\modp#1{\pmod#1}
\global\long\def\sequence#1#2{\left\langle #1\,\middle|\,#2\right\rangle }
\global\long\def\set#1#2{\left\{  #1\,\middle|\,#2\right\}  }
\global\long\def\Diag{\operatorname{Diag}}
\global\long\def\Nn{\mathbb{N}}
\global\long\def\mathrela#1{\mathrel{#1}}
\global\long\def\twiddle{\mathord{\sim}}
\global\long\def\mathordi#1{\mathord{#1}}
\global\long\def\Qq{\mathbb{Q}}
\global\long\def\dense{\operatorname{dense}}
 \global\long\def\cof{\operatorname{cof}}
\global\long\def\tr{\operatorname{tr}}
\global\long\def\treeexp#1#2{#1^{\left\langle #2\right\rangle _{\tr}}}
\global\long\def\x{\times}
\global\long\def\forces{\Vdash}
\global\long\def\Vv{\mathbb{V}}
\global\long\def\Zz{\mathbb{Z}}
\global\long\def\tauname{\dot{\tau}}
\global\long\def\ScottPsi{\Psi}
\global\long\def\cont{2^{\aleph_{0}}}
\global\long\def\MA#1{{MA}_{#1}}
\global\long\def\rank#1#2{R_{#1}\left(#2\right)}
\global\long\def\cal#1{\mathcal{#1}}
\global\long\def\Ff{\mathbb{F}}

\def\Ind#1#2{#1\setbox0=\hbox{$#1x$}\kern\wd0\hbox to 0pt{\hss$#1\mid$\hss} \lower.9\ht0\hbox to 0pt{\hss$#1\smile$\hss}\kern\wd0} 
\def\Notind#1#2{#1\setbox0=\hbox{$#1x$}\kern\wd0\hbox to 0pt{\mathchardef \nn="3236\hss$#1\nn$\kern1.4\wd0\hss}\hbox to 0pt{\hss$#1\mid$\hss}\lower.9\ht0 \hbox to 0pt{\hss$#1\smile$\hss}\kern\wd0} 
\def\nind{\mathop{\mathpalette\Notind{}}}

\global\long\def\ind{\mathop{\mathpalette\Ind{}}}
\global\long\def\opp{\operatorname{opp}}
 \global\long\def\nind{\mathop{\mathpalette\Notind{}}}
\global\long\def\average#1#2#3{Av_{#3}\left(#1/#2\right)}
\global\long\def\mx#1{Mx_{#1}}
\global\long\def\maps{\mathfrak{L}}

\global\long\def\Esat{E_{\mbox{sat}}}
\global\long\def\Ebnf{E_{\mbox{rep}}}
\global\long\def\Ecom{E_{\mbox{com}}}
\global\long\def\BtypesA{S_{\Bb}^{x}\left(A\right)}

\global\long\def\init{\trianglelefteq}
\global\long\def\fini{\trianglerighteq}
\global\long\def\Bb{\cal B}
\global\long\def\Lim{\operatorname{Lim}}
\global\long\def\Succ{\operatorname{Succ}}

\global\long\def\SquareClass{\cal M}
\global\long\def\leqstar{\leq_{*}}
\global\long\def\average#1#2#3{Av_{#3}\left(#1/#2\right)}
\global\long\def\cut#1{\mathfrak{#1}}
\global\long\def\Prime{Pr}
\global\long\def\TheoryofPrimes{T_{+,\Prime}}
\global\long\def\TheoryofPrimesOrder{T_{+,\Prime,<}}

\global\long\def\OurSequence{\mathcal{I}}

\title{Decidability and classification of the theory of integers with primes}

\author{Itay Kaplan and Saharon Shelah}

\thanks{The first author would like to thank the Israel Science foundation
for partial support of this research (Grant no. 1533/14). }

\thanks{The research leading to these results has received funding from the
European Research Council, ERC Grant Agreement n. 338821. No. 1082
on the second author's list of publications.}

\address{Itay Kaplan \\
The Hebrew University of Jerusalem\\
Einstein Institute of Mathematics \\
Edmond J. Safra Campus, Givat Ram\\
Jerusalem 91904, Israel}

\email{kaplan@math.huji.ac.il}

\urladdr{https://sites.google.com/site/itay80/ }

\address{Saharon Shelah\\
The Hebrew University of Jerusalem\\
Einstein Institute of Mathematics \\
Edmond J. Safra Campus, Givat Ram\\
Jerusalem 91904, Israel}

\address{Saharon Shelah \\
Department of Mathematics\\
Hill Center-Busch Campus\\
Rutgers, The State University of New Jersey\\
110 Frelinghuysen Road\\
Piscataway, NJ 08854-8019 USA}

\email{shelah@math.huji.ac.il}

\urladdr{http://shelah.logic.at/}

\subjclass[2010]{03C45, 03F30, 03B25, 11A41}

\keywords{Model theory, Decidability, Primes, Dickson's conjecture}
\begin{abstract}
We show that under Dickson's conjecture about the distribution of
primes in the natural numbers, the theory $Th\left(\Zz,+,1,0,\Prime\right)$
where $\Prime$ is a predicate for the prime numbers and their negations
is decidable, unstable and supersimple. This is in contrast with $Th\left(\Zz,+,0,\Prime,<\right)$
which is known to be undecidable by the works of Jockusch, Bateman
and Woods. 
\end{abstract}

\maketitle

\section{Introduction}

It is well known that Presburger arithmetic $T_{+,<}=Th\left(\Zz,+,0,1,<\right)$
is decidable and enjoys quantifier elimination after introducing predicates
for divisibility by $n$ for every natural number $n>1$ (see e.g.,
\cite[Corollary 3.1.21]{Marker}). The same is true for $T_{+}=Th\left(\Zz,+,0,1\right)$.
This is, of course, in contrast to the situation with the theory of
Peano arithmetics or $Th\left(\Zz,+,\cdot,0,1\right)$ which is not
decidable. 

If we are interested in classifying these theories in terms of stability
theory, quantifier elimination gives us that $T_{+}$ is superstable
of $U$-rank $1$, while $T_{+,<}$ is dp-minimal (a subclass of dependent,
or NIP, theories, see e.g., \cite{MR2822489,Simon-Dp-min,OnUs1}). 

Over the years there has been quite extensive research on structures
with universe $\Zz$ or $\Nn$ and some extra structure, usually definable
from Peano. A very good survey regarding questions of decidability
is \cite{MR1900397} and a list of such structures defining addition
and multiplication is available in \cite{MR1824851}. 

Less research was done on classifying these structures stability-theoretically.
For instance, in \cite[Theorem 25]{Pois45} and also in \cite{PalacinSklinos2014}
it is proved that $Th\left(\Zz,+,0,P_{q}\right)$ is superstable of
$U$-rank $\omega$, where $P_{q}$ is the set of powers of $q$. 

In this paper we are interested in adding a predicate $\Prime$ for
the primes and their negations and we consider $\TheoryofPrimes=Th\left(\Zz,+,0,1,\Prime\right)$
and $\TheoryofPrimesOrder=Th\left(\Zz,+,0,1,\Prime,<\right)$. The
language $\left\{ +,0,1,\Prime\right\} $ allows us to express famous
number-theoretic conjectures such as the twin prime conjecture (for
every $n$, there are at least $n$ pairs of primes/negation of primes
of distance $2$), and a version of Goldbach's conjecture (all even
integers can be expressed as a difference or a sum of primes). Adding
the order allows us to express Goldbach's conjecture in full. 

Up to now, the only known results about the theory are under a strong
number-theoretic conjecture known as Dickson conjecture (D) (see below),
which is also the assumption in the works of Jockusch, Bateman and
Woods. In \cite{Batemanetc,WoodsThesis}, they proved that assuming
Dickson conjecture, $Th\left(\Nn,+,0,\Prime\right)$ is undecidable
and even defines multiplication. It follows immediately that $\TheoryofPrimesOrder$
is undecidable and as complicated as possible in the sense of stability
theory. This also explains why we need $\Prime$ to include also the
negation of primes: by relatives of the Goldbach Conjecture (which
are proved, see e.g., \cite{Tao2013}), every positive integer greater
than $N$ is a sum of $K$ primes for some fixed $K,N$, and hence
the positive integers themselves are also definable from the positive
primes. 
\begin{conjecture}
[\bf{D}] (Dickson, 1904 \cite{dickson}) Let $k\geq1$ and $\bar{f}=\sequence{f_{i}}{i<k}$
where $f_{i}\left(x\right)=a_{i}x+b_{i}$ with $a_{i},b_{i}$ non-negative
integers, $a_{i}\geq1$ for all $i<k$. Assume that the following
condition holds:
\begin{itemize}
\item [$\star_{\bar{f}}$]There does not exist any integer $n>1$ dividing
all the products $\prod_{i<k}f_{i}\left(s\right)$ for every (non-negative)
integer $s$. 
\end{itemize}
Then there exist infinitely many natural numbers $m$ such that $f_{i}\left(m\right)$
is prime for all $i<k$. 
\end{conjecture}
Note that in fact the condition $\star_{\bar{f}}$ follows easily
from the conclusion that there are infinitely many $m$'s with $f_{i}\left(m\right)$
prime for all $i<k$.  See also Remark \ref{rem:eliminates infinity}.

For a discussion of this conjecture see \cite{MR1016815}. 

Our main result is the following.
\begin{thm}
\label{thm:main}Assuming (D), the theory $T_{+,\Prime}$ is decidable,
unstable and supersimple of $U$-rank $1$.
\end{thm}
In essence (D) implies that the set of primes is generic up to congruence
conditions (while it is not generic in the sense of \cite{MR1650667}),
and this allows us to get quantifier elimination in a suitable language.
Forking then turns out to be trivial: forking formulas are algebraic
(Theorem \ref{thm:forking is trivial}).

To show that $\TheoryofPrimes$ is unstable we show that it has the
independence property (see Proposition \ref{prop:arithmatic progressions}).
This turns out to follow from the proof of the Green-Tao theorem about
arithmetic progressions in the primes \cite{GreenTao} (i.e., without
using (D)), as was told to us in a private communication by Tamar
Ziegler (but we also show that this follows from (D)).

\subsection*{Acknowledgments}

\thanks{We would like to thank Tamar Ziegler for telling us about (D) and
for her input on the Green-Tao theorem (see Proposition \ref{prop:arithmatic progressions}). }

\thanks{We would also like to thank Carl Jockusch, Philipp Hieronymi, Lou
van den Dries and Alexis B\`es for reassuring us that the results
stated here are new. }

\section{Quantifier elimination }

In this section we will prove quantifier elimination in $\TheoryofPrimes$
assuming (D) in a suitable language.

Let us first note some useful facts about (D). 
\begin{rem}
\label{rem:bounding the integer}Given a sequence of linear maps $\sequence{f_{i}}{i<k}$
where $f_{i}\left(x\right)=a_{i}x+b_{i}$ as in (D), $\star_{\bar{f}}$
holds iff for every prime $p<N$, $p$ does not divide $\prod_{i<k}f_{i}\left(s\right)$
for all $s\in\Zz$ where $N=\max\left(\set{a_{i}}{i<k}\cup\left\{ k\right\} \right)+1$. \end{rem}
\begin{proof}
If $\star_{\bar{f}}$ fails, then there is some prime $p$ such that
$p$ divides $\prod_{i<k}f_{i}\left(s\right)$ for all $s$. Let $P\left(X\right)\in\Zz\left[X\right]$
be the polynomial $\prod_{i<k}f_{i}\left(X\right)$. Let $P_{p}=P\modp p\in\Ff_{p}\left[x\right]$
(where $\Ff_{p}$ is the prime field of size $p$). It follows that
$P_{p}\left(a\right)=0$ for all $a\in\Ff_{p}$. So either $P_{p}=0$
or $k\geq\deg\left(P_{p}\right)\geq p$,  hence $p\leq k$ or $\prod_{i<k}a_{i}\equiv0\modp p$
(as the leading coefficient) which means that for some $i<k$, $a_{i}\geq p$,
so $p<N$ and we are done. \end{proof}
\begin{lem}
\label{lem:b_i can be negative}Assume (D). Then (D) holds also when
we allow $b_{i}$ to be negative.\end{lem}
\begin{proof}
Suppose that $\sequence{f_{i}}{i<k}$ is a sequence of linear maps
$f_{i}\left(x\right)=a_{i}x+b_{i}$ where $a_{i}\geq1$ and $b_{i}\in\Zz$,
and assume that $\star_{\bar{f}}$ holds. Let $N$ be as in Remark
\ref{rem:bounding the integer}. Let $K=N!$ (enough to take the product
of the primes below $N$). Suppose that $l\in\Nn$ is such that $lK+b_{i}>0$
for all $i<k$. Let $f_{i}'\left(x\right)=a_{i}x+a_{i}lK+b_{i}$.
Then $a_{i}\geq1$, $b_{i}'=a_{i}lK+b_{i}>0$, so let us show that
$\star_{\bar{f}'}$ holds (where $\bar{f}'=\sequence{f_{i}'}{i<k}$).
Note that when we compute $N$ in Remark \ref{rem:bounding the integer},
we only use $k$ and $a_{i}$ which haven't changed, so by that remark,
it is enough to check that for no prime $p<N$, $\prod_{i<k}f'_{i}\left(s\right)\equiv0\modp p$
for all $s$. But for such $p$'s, $f_{i}'\left(s\right)=f_{i}\left(s\right)+a_{i}lK\equiv f_{i}\left(s\right)\modp p$,
so $\prod_{i<k}f_{i}'\left(s\right)\equiv\prod_{i<k}f_{i}\left(s\right)\not\equiv0\modp p$. 

By (D), there are infinitely many integers $m$ such that $f_{i}'\left(m\right)$
is prime for all $i<k$. But $f_{i}'\left(m\right)=a_{i}m+a_{i}lK+b_{i}=a_{i}\left(m+lK\right)+b_{i}$.
Hence substituting $m+lK$ for $m$ we get what we wanted. \end{proof}
\begin{lem}
\label{lem:Also composite}Assuming (D), given $f_{i}\left(x\right)=a_{i}x+b_{i}$
with $a_{i},b_{i}$ integers, $a_{i}\geq1$ for all $i<k$ and $g_{j}\left(x\right)=c_{j}x+d_{j}$
with $c_{j},d_{j}$ integers, $c_{j}\geq1$ for all $j<k'$, if $\star_{\bar{f}}$
holds for $\bar{f}=\sequence{p_{i}}{i<k}$ and $\left(a_{i},b_{i}\right)\neq\left(c_{j},d_{j}\right)$
for all $i,j$ then there are infinitely many natural numbers $m$
for which $f_{i}\left(m\right)$ is prime and $g_{j}\left(m\right)$
is composite for all $i<k,j<k'$. 
\end{lem}
Before giving the proof, we note that this lemma generalizes Lemma
1 from \cite{Batemanetc}, which was key in the proof there of the
undecidability of $\TheoryofPrimesOrder$. 
\begin{cor}
\cite[Lemma 1]{Batemanetc}(Assuming (D)) Let $b_{0},\ldots,b_{n-1}$
be an increasing sequence of natural numbers, and assume that there
is no prime $p$ such that $\set{b_{i}\modp p}{i<n}=p$. Then there
are infinitely many natural numbers $x$ such that $x+b_{0},\ldots,x+b_{n-1}$
are \textbf{consecutive} primes. \end{cor}
\begin{proof}
[Proof of Corollary] This is immediate from Lemma \ref{lem:Also composite}
by taking $f_{i}\left(x\right)=x+b_{i}$ and $g_{j}\left(x\right)=x+c_{j}$
where $c_{j}$ run over all numbers between the $b_{j}$'s. 
\end{proof}

\begin{proof}
[Proof of Lemma]By induction on $k'$. For $k'=0$ there is nothing
to prove by (D) and Lemma \ref{lem:b_i can be negative}. 

Suppose the lemma is true for $k'$ and prove it for $k'+1$. It is
enough to prove that for any $n$, there is some $m>n$ such that
$f_{i}\left(m\right)$ is prime for all $i<k$ and $g_{j}\left(m\right)$
is not prime for all $j<k'$. 

Fix $n$. We may assume by enlarging it that for no $m>n$ is it the
case that $f_{i}\left(m\right)=g_{j}\left(m\right)$ for $i<k,j\leq k'$. 

Let $m>n$ be so that $f_{i}\left(m\right)$ is prime for all $i<k$
and $g_{j}\left(m\right)$ is composite for all $j<k'$. If it happens
that $g_{k'}\left(m\right)$ is composite, then we are done, so suppose
that $q=g_{k'}\left(m\right)$ is prime. Let $f_{i}'\left(x\right)=a_{i}\left(qx+m\right)+b_{i}$
and $g_{j}'\left(x\right)=c_{j}\left(qx+m\right)+d_{j}$ for $i<k$
and $j<k'+1$. Then $g_{k'}'\left(x\right)=c_{j}qx+q$ is composite
for all $x\geq1$ (so that $c_{j}x+1\geq2$). Hence it is enough to
find $m'$ large enough so that $f_{i}'\left(m'\right)$ is prime
for all $i<k$ and $g_{j}'\left(m'\right)$ is composite for all $j<k'$. 

By the induction hypothesis, it is enough to check that $\star_{\bar{f}'}$
holds (because $\left(a_{i}q,a_{i}m+b_{i}\right)\neq\left(c_{j}q,c_{j}m+d_{j}\right)$).
Suppose that $p>1$ is a prime which divides $\prod_{i<k}f_{i}'\left(s\right)$
for all $s$. Hence $\prod_{i<k}f_{i}'\left(s\right)\equiv0\modp p$,
and if $p\neq q$, it follows (as $q$ is invertible modulo $p$)
that $\prod_{i<k}f_{i}\left(s\right)\equiv0\modp p$ for all $s$
--- a contradiction. If $p=q$, then $f_{i}'\left(x\right)\equiv a_{i}m+b_{i}\equiv f_{i}\left(m\right)\modp q$
for all $x$, hence for some $i<k$, $f_{i}\left(m\right)=q=g_{k'}\left(m\right)$,
contradicting our choice of $m$. 
\end{proof}
Expand the language $L=\left\{ +,\Prime,0,1\right\} $ to include
the Presburger predicates $P_{n}$ for $2\leq n<\omega$ interpreted
as $P_{n}\left(x\right)\Leftrightarrow x\equiv0\modp n$, and also
the predicates $\Prime_{n}$ for $2\leq n<\omega$ interpreted as
$\Prime_{n}\left(x\right)\Leftrightarrow P_{n}\left(x\right)\land\Prime\left(x/n\right)$.
We need the latter predicate in order to eliminate the quantifiers
from $\varphi\left(x\right)=\exists y\left(ny=x\land\Prime\left(y\right)\right)$.
We also add negation (as a unary function). We need negation because
of formulas of the form $\varphi\left(x,y\right)=\Prime\left(x-y\right)=\exists w\left(w+y=x\land\Prime\left(w\right)\right)$.

Let $L^{*}$ be the resulting language $\set{+,-,1,0,\Prime,\Prime_{n},P_{n}}{2\leq n<\omega}$,
and let $\TheoryofPrimes^{*}$ be the complete theory of $M^{*}$
--- the structure with universe $\Zz$ in $L^{*}$. Note that all
the new predicates are definable from $L$. 
\begin{rem}
\label{rem:first order}The condition $\star_{\bar{p}}$ of Dickson's
conjecture is first-order expressible in $L^{*}$. This means that
for every tuple $a_{i},i<k$ of positive integers, there is a formula
$\varphi_{\bar{a}}\left(y_{0},\ldots,y_{k-1}\right)$ such that for
any choice of $b_{i}\in\Zz$ for $i<k$, $M^{*}\models\varphi_{\bar{a}}\left(\bar{b}\right)$
iff $\star_{\bar{f}}$ holds where $f_{i}\left(x\right)=a_{i}x+b_{i}$. \end{rem}
\begin{proof}
Recall Remark \ref{rem:bounding the integer} and the choice of $N$
from there (which depends only on $\sequence{a_{i}}{i<k}$ and $k$).
Let $\varphi_{\bar{a}}\left(\bar{y}\right)$ say that for every prime
$p<N$, for some $0\leq x<p$, for all $i<k$, $\neg P_{p}\left(a_{i}x+y_{i}\right)$.
Note that $\varphi_{\bar{a}}$ is quantifier-free in $L^{*}$ (as
it contains $1$). \end{proof}
\begin{rem}
\label{rem:eliminates infinity} Given $\bar{f}=\sequence{f_{i}}{i<k}$
as in Remark \ref{rem:bounding the integer}, if there are more than
$2k$ integers $m$ such that $f_{i}\left(m\right)$ is prime or a
negation of a prime, then $\star_{\bar{f}}$ holds. Indeed, otherwise
there is some prime $p$ which witnesses this, but then for some $i$
and three different $m$'s, $\left|p_{i}\left(m\right)\right|=p$
--- a contradiction. \end{rem}
\begin{lem}
\label{lem:Elimination of Quantifiers}$\TheoryofPrimes^{*}$ eliminates
quantifiers in $L^{*}$ provided (D). \end{lem}
\begin{proof}
We start with the following observation.
\begin{itemize}
\item [$\diamondsuit$] By Remark \ref{rem:first order} and Lemma \ref{lem:Also composite},
our assumption that Dickson's conjecture holds translates into a first
order statement: for every $n$ and every choice of positive integers
$\sequence{a_{i}}{i<k}$ and $\sequence{a_{j}'}{j<k'}$ and for all
$\sequence{b_{i}}{i<k}$ and $\sequence{b_{j}'}{j<k'}$, if $\varphi_{\bar{a}}\left(\bar{b}\right)$
holds and $\left(a_{i},b_{i}\right)\neq\left(a_{j}',b_{j}'\right)$
for all $i,j$ then there are at least $n$ elements $x$ with $\bigwedge_{i<k}\Prime\left(a_{i}x+b_{i}\right)\land\bigwedge_{i<k'}\neg\Prime\left(a_{j}'x+b'_{j}\right)$.
Conversely, By Remark \ref{rem:eliminates infinity}, if there are
more than $2k$ such elements $x$, then $\varphi_{\bar{a}}\left(\bar{b}\right)$
holds. Together, $\varphi_{\bar{a}}\left(\bar{b}\right)\wedge\bigwedge_{i,j}\left(a_{i},b_{i}\right)\neq\left(a_{j}',b_{j}'\right)$
holds iff there are more than $2k$ elements $x$ with 
\[
\bigwedge_{i<k}\Prime\left(a_{i}x+b_{i}\right)\land\bigwedge_{i<k'}\neg\Prime\left(a_{j}'x+b_{j}'\right).
\]
(Recall that $\Prime$ is contains the primes and their negations.)
\end{itemize}
In order to prove quantifier elimination we will use a back-and-forth
criteria. Namely, suppose that $\C\models\TheoryofPrimes^{*}$ is
a monster model (very large, saturated model) and that $h:A\to B$
is an isomorphism of small substructures $A,B$. Given $a\in\C\backslash A$
we want to extend $h$ so that its domain contains $a$.

We may assume, by our choice of language (which includes $\Prime_{n}$
and $-$), that both $A$ and $B$ are groups such that if $c\in A$
and $\C\models P_{n}\left(a\right)$ then $c/n\in A$ and similarly
for $B$. Why? for such a $c$, elements of the group generated by
adding $c/n$ to $A$ have the form $m\left(c/n\right)+b$ for $m\in\Zz$
and $b\in A$. We have to show that the map taking $c/n$ to $h\left(c\right)/n$
and extends $h$ is an isomorphism. For instance, we have to show
that if $\C\models\Prime\left(m\left(c/n\right)+b\right)$ then $\C\models\Prime\left(m\left(h\left(c\right)/n\right)+h\left(b\right)\right)$.
But $\C\models\Prime\left(m\left(c/n\right)+b\right)$ iff $\C\models\Prime_{n}\left(mc+nb\right)$.
Similarly we deal with $\Prime_{m}$ and $P_{m}$.  

Let $p^{a,A,h}\left(x\right)=\tp^{\qf}\left(a/A\right)$, and let
$q^{a,A,h}\left(x\right)=h\left(p^{a,A,h}\right)$. Let $p_{\equiv}^{a,A,h}=p^{a,A,h}\upharpoonright L_{\equiv}^{*}$
and $p_{\Prime}^{a,A,h}=p^{a,A,h}\upharpoonright L_{\Prime}^{*}$,
where $L_{\equiv}^{*}=L^{*}\backslash\set{\Prime,\Prime_{n}}{2\leq n<\omega}$
and $L_{\Prime}^{*}=L^{*}\backslash\set{P_{n}}{2\leq n<\omega}$,
so that $p^{a,A,h}=p_{\equiv}^{a,A,h}\cup p_{\Prime}^{a,A,h}$, and
we have to realize $q^{a,A,h}$. 
\begin{claim}
\label{claim:no need for modulo}It is enough to prove that we can
realize $q_{\Prime}^{a,A,h}=h\left(p_{\Prime}^{a,A,h}\right)$ for
all $A,a$ and $h$ as above.\end{claim}
\begin{proof}
Easily, as we included $1$ in the language, $q_{\equiv}^{a,A,h}$
is isolated by $\set{x\neq c}{c\in B}$  and equations of the form
$x\equiv k\modp n$ for $k<n$,  and for every $n<\omega$ there
is exactly one $k<n$ with such an equation appearing in $q^{a,A,h}$.
Also, every finite set of such equations is implied by one such equation
(e.g., if the equations are $\set{x\equiv k_{i}\modp{n_{i}}}{i<s}$
then take $x\equiv k\modp{\prod_{i<s}n_{i}}$ where $k$ is such that
this equation is in $q^{a,A,h}$). Hence it is enough to show that
$x\equiv k\modp n\cup q_{\Prime}^{a,A,h}\left(x\right)$ is consistent
($q_{\Prime}^{a,A,h}$ already contains $\set{x\neq c}{c\in B}$).
As $a\equiv k\modp n$, $b=\left(a-k\right)/n\in\C$. Let $p^{b,A,h}=\tp^{\qf}\left(b/A\right)$
so by our assumption there is some $d\in\C$ such that $d\models h\left(p^{b,A,h}\right)_{\Prime}$.
Then $nd+k\models q_{\Prime}^{a,A,h}\left(x\right)$ and of course
satisfies the equation $x\equiv k\modp n$. 
\end{proof}
Let $p_{\Prime_{0}}^{a,A,h}=p^{a,A,h}\upharpoonright L_{\Prime_{0}}$
where $L_{\Prime_{0}}=L_{\Prime}\backslash\set{\Prime_{n}}{2\leq n<\omega}$.
\begin{claim}
It is enough to prove that we can realize $q_{\Prime_{0}}^{a,A,h}=h\left(p_{\Prime_{0}}^{a,A,h}\right)$
for all $A,a$ and $h$ as above. \end{claim}
\begin{proof}
This is similar to Claim \ref{claim:no need for modulo}. It is enough
to show that $q_{\Prime_{0}}^{a,A,h}\left(x\right)\cup\Sigma\left(x\right)$
is consistent where $\Sigma$ is a finite set of formulas from $q_{\Prime}^{a,A,h}\backslash q_{\Prime_{0}}^{a,A,h}$.
So $\Sigma$ consists of formulas of the form $\Prime_{n}\left(mx+c\right)$
or its negation for $m\in\Zz$, $1<n\in\Nn$ and $c\in B$. Without
loss of generality, by replacing the $n$'s with their product $N$
and $\Prime_{n}\left(mx+c\right)$ by $\Prime_{N}\left(\left(N/n\right)\left(mx+c\right)\right)$,
we may assume that all the $n$'s appearing in $\Sigma$ are equal
to $n>1$. Let $b=\left(a-k\right)/n$ where $a\equiv k\modp n$ and
$k<n$. Let $p^{b,A,h}=\tp^{\qf}\left(b/A\right)$. By our assumption
there is some $d\in\C$ such that $d\models h\left(p^{b,A,h}\right)_{\Prime_{0}}$.
Let us check that $nd+k\models q_{\Prime_{0}}^{a,A,h}\left(x\right)\cup\Sigma\left(x\right)$. 

First, if $\varphi\left(x,c\right)\in q_{\Prime_{0}}^{a,A,h}\left(x\right)$
($c$ a tuple from $B$) then $\C\models\varphi\left(a,h^{-1}\left(c\right)\right)$
so that $\C\models\varphi\left(nb+k,h^{-1}\left(c\right)\right)$
so $d\models\varphi\left(nx+k,c\right)$ so $nd+k\models\varphi\left(x,c\right)$.

Now, suppose that $\Prime_{n}\left(mx+c\right)\in\Sigma$. 

Then $\C\models\Prime_{n}\left(ma+h^{-1}\left(c\right)\right)$, so
$\C\models\Prime_{n}\left(m\left(nb+k\right)+h^{-1}\left(c\right)\right)$.
Hence $m\left(nb+k\right)+h^{-1}\left(c\right)$ is divisible by $n$
which means that $mk+h^{-1}\left(c\right)$ is divisible by $n$,
and as $h$ is an isomorphism (and the language includes $1$), so
is $mk+c$, hence $m\left(nd+k\right)+c$ is also divisible by $n$.
Moreover the quotient $e=\left[mk+h^{-1}\left(c\right)\right]/n\in A$
maps to $e'=\left[mk+c\right]/n\in B$. As $\C\models\Prime\left(mb+e\right)$,
it follows that $\C\models\Prime\left(md+e'\right)$, so that $\C\models\Prime_{n}\left(m\left(nd+k\right)+c\right)$.
The same logic works if $\neg\Prime_{n}\left(mx+c\right)\in\Sigma$. 
\end{proof}
Divide into cases.
\begin{casenv}
\item There are infinitely many solutions to $p_{\Prime_{0}}^{a,A,h}$.

Given any finite set $\Sigma\subseteq q_{\Prime_{0}}$, it has the
form 
\[
\set{\Prime\left(m_{i}x+c_{i}\right)}{i<k}\cup\set{\neg\Prime\left(m'_{j}x+c'_{j}\right)}{j<k'}
\]
 where $m_{i},m_{j}'\in\Zz$ and $c_{i},c_{j}'\in B$ (it also includes
formulas of the form $x\neq c$). As\footnote{Here we use the fact that $\Prime$ contains both the primes and their
negations.} $\C\models\forall x\Prime\left(x\right)\leftrightarrow\Prime\left(-x\right)$,
we may assume that $m_{i},m_{j}'\geq1$. Also, it is of course impossible
that $\left(m_{i},c_{i}\right)=\left(m_{j}',c'_{j}\right)$. By $\diamondsuit$,
it is enough to check that $\C\models\varphi_{\bar{m}}\left(\bar{c}\right)$
where $\bar{m}=\sequence{m_{i}}{i<k}$ and $\bar{c}=\sequence{c_{i}}{i<k}$
and $\varphi_{\bar{m}}$ is from Remark \ref{rem:first order}. As
$\varphi_{\bar{m}}$ is quantifier-free, and as $\C\models\varphi_{\bar{m}}\left(h^{-1}\left(\bar{c}\right)\right)$
(because $h^{-1}\left(\Sigma\right)$ has infinitely many solutions
and by $\diamondsuit$), we are done. 

\item There are only finitely many solutions to $p_{\Prime_{0}}$.

By $\diamondsuit$, and as $\C\models\forall x\Prime\left(x\right)\leftrightarrow\Prime\left(-x\right)$,
there are some $m_{i}\geq1,e_{i}\in A$ such that $\set{\Prime\left(m_{i}x+e_{i}\right)}{i<k}$
already has finitely many solutions. Hence $\varphi_{\bar{m}}\left(\bar{e}\right)$
fails, so for some $p<N$ (see Remark \ref{rem:first order}), there
is some $i<k$ such that $P_{p}\left(m_{i}a+e_{i}\right)$. But as
$\Prime\left(m_{i}a+e_{i}\right)$, it must be that $\pm p=m_{i}a+e_{i}$.
As $\pm p,e_{i}\in A$, and as $A$ is closed under dividing by $m_{i}$,
it follows that $a\in A$, and we are done. 

\end{casenv}
\end{proof}

\section{\label{sec:Classifying}Decidability and classification}

We start with the decidability result that is now almost immediate. 
\begin{cor}
The theory $\TheoryofPrimes^{*}$ is decidable and hence so is $\TheoryofPrimes$
provided that Dickson's conjecture holds.\end{cor}
\begin{proof}
Observing the proof of Lemma \ref{lem:Elimination of Quantifiers},
we see that we can recursively enumerate the axioms that we used.
Let us denote this set by $\Sigma$. Let $\Sigma'$ be the complete
quantifier-free theory of $\Zz$ in $L^{*}$. Then $\Sigma'$ is recursive
and contained in $\TheoryofPrimes^{*}$. 

Then the proof gives us that if $M_{1},M_{2}$ are two saturated models
of $\Sigma\cup\Sigma'$, then they are isomorphic (start with $A,B$
being the structures generated by $1$ in $M_{1},M_{2}$ respectively).
This implies that $\Sigma\cup\Sigma'$ is complete and hence decidable.

\end{proof}
Now we turn to classification in the sense of \cite{Sh:c}, where
one is interested in classifying theories by finding ``classes''
having interesting properties in the class and outside of it. The
most studied such class is that of stable theories, which is a very
well-behaved and well-understood class. Containing it is the class
of simple theories, and among them the ``simplest'' simple theories
are supersimpe of $U$-rank 1. For the definition of simple and supersimple
theories we refer the reader to e.g., \cite[Chapter 7, Definition 8.6.3]{TentZiegler}. 
\begin{thm}
\label{thm:forking is trivial}Assuming (D), $\TheoryofPrimes^{*}$
(and $\TheoryofPrimes$) is supersimple of $U$-rank $1$: if $\varphi\left(x,a\right)$
forks over $A$ where $x$ is a singleton and $a$ is some tuple from
$A$ then $\varphi$ is algebraic (i.e., $\varphi\vdash\bigvee_{i<k}x=c_{i}$).
\end{thm}
\begin{proof}
The proof is similar to that of Lemma \ref{lem:Elimination of Quantifiers}. 

Let $N\supseteq A$ be an $\left|A\right|^{+}$-saturated model. Suppose
that $\varphi$ forks over $A$ but is not algebraic. Extend $\varphi$
to a type $p\left(x\right)\in S\left(N\right)$ which is non-algebraic
over $N$. So $p$ forks over $A$, and hence it divides over $A$.
Hence it divides over $\acl\left(A\right)$ (see e.g., \cite[Proof of Lemma 3.21]{Kachernikov}),
so we may assume that $A=\acl\left(A\right)$. By quantifier elimination
we may assume that $p$ is quantifier free. 

Recalling the notation from the proof of Lemma \ref{lem:Elimination of Quantifiers},
we have the following claim. 
\begin{claim}
\label{claim:Reduction to modulo}It is enough to prove that for every
type $q\left(x\right)\in S\left(N\right)$, if $q_{\Prime}=q\upharpoonright L_{\Prime}^{*}$
divides over $A$, then $q_{\Prime}$ is algebraic. \end{claim}
\begin{proof}
We want to show that $p$ is algebraic, thus getting a contradiction.
Let $\sequence{N_{i}}{i<\omega}$ be an indiscernible sequence over
$A$ starting with $N_{0}=N$ in $\C$, which witnesses that $p$
divides over $A$. 

Let $p_{\equiv}=p\upharpoonright L_{\equiv}^{*}$. 

As $p_{\equiv}|_{A}|\vdash p_{\equiv}\vdash\bigcup\set{p_{\equiv}\left(x,N_{i}\right)}{i<\omega}$,
it follows that $\bigcup\set{p_{\Prime}\left(x,N_{i}\right)}{i<\omega}\cup\Sigma$
is inconsistent for some finite $\Sigma$, which is isolated by a
formula of the form $x\equiv k\modp n$ for some $k<n$. 

Let $c\models p$. Then $c\equiv k\modp n$, and let $d=\left(c-k\right)/n$.
Then $\left[\tp\left(d/N\right)\right]_{\Prime}$ divides over $A$
as witnessed by the same sequence $\sequence{N_{i}}{i<\omega}$ (let
$r=\tp\left(d/N\right)$, then if $d'\models\bigcup\set{r_{\Prime}\left(x,N_{i}\right)}{i<\omega}$
then $nd'+k\models\Sigma\cup\bigcup\set{p_{\Prime}\left(x,N_{i}\right)}{i<\omega}$).
Hence, $\left[\tp\left(d/N\right)\right]_{\Prime}$ is algebraic,
i.e., $d\in N$, but then so is $c$. \end{proof}
\begin{claim}
\label{claim:reduction to pr0}It is enough to prove that for every
type $q\left(x\right)\in S\left(N\right)$, if $q_{\Prime_{0}}=q\upharpoonright L_{\Prime_{0}}^{*}$
divides over $A$, then $q_{\Prime_{0}}$ is algebraic.\end{claim}
\begin{proof}
This is similar to the proof of Claim \ref{claim:Reduction to modulo}. 

By Claim \ref{claim:Reduction to modulo}, it is enough to prove that
for any $q\left(x\right)\in S\left(N\right)$, if $q_{\Prime}$ divides
over $A$ then $q_{\Prime}$ is algebraic. Suppose that $q_{\Prime}$
divides over $A$ and let $\sequence{N_{i}}{i<\omega}$ be as in the
proof of Claim \ref{claim:Reduction to modulo}. There is some finite
set of formulas $\Sigma\left(x,N\right)\subseteq q_{\Prime}\backslash q_{\Prime_{0}}$
such that $\bigcup\set{q_{\Prime_{0}}\left(x,N_{i}\right)\cup\Sigma\left(x,N_{i}\right)}{i<\omega}$
is inconsistent. As in the proof of Lemma \ref{lem:Elimination of Quantifiers},
we may assume that for some $n\in\Nn$, $\Sigma$ consists of formulas
of the form $\Prime_{n}\left(mx+c\right)$ for $c\in N$ and $m\in\Zz$.
Let $d\models q$, and assume that $d\equiv k\modp n$ for $k<n$.
Then for some $e\in\C$, $d=ne+k$, and $\left[\tp\left(e/N\right)\right]_{\Prime_{0}}$
divides over $A$ (let $r=\tp\left(e/N\right)$, then if $e'\models\bigcup\set{r_{\Prime_{0}}\left(x,N_{i}\right)}{i<\omega}$
then $ne'+k\models\bigcup\set{q_{\Prime_{0}}\left(x,N_{i}\right)\cup\Sigma\left(x,N_{i}\right)}{i<\omega}$,
as in the proof of Lemma \ref{lem:Elimination of Quantifiers}). Hence
this type is algebraic and hence so is $q$. \end{proof}
\begin{claim}
\label{claim:finite set of pr formulas}It is enough to prove that
if $\Sigma\left(x\right)$ is a finite set of formulas of the form
$\Prime\left(mx+c\right)$ or $\neg\Prime\left(mx+c\right)$ for $m\in\Zz$
and $c\in N$, which has infinitely many solutions, then $\Sigma$
does not divide over $A$.\end{claim}
\begin{proof}
Use Claim \ref{claim:reduction to pr0}. We have to prove that if
$q_{\Prime_{0}}$ divides over $A$ then it is algebraic. Suppose
it is not, and let $\Sigma\subseteq q_{\Prime_{0}}$ be finite such
that $\Sigma\left(x\right)\cup\set{x\neq c}{c\in N}$ divides over
$A$. Then $\Sigma$ has infinitely many solutions and is of the right
form, so we are done. 
\end{proof}
Let $\Sigma\left(x\right)$ be as in Claim \ref{claim:finite set of pr formulas}. 

Then $\Sigma\left(x,\bar{c},\bar{c}'\right)=\set{\Prime\left(m_{i}x+c_{i}\right)}{i<k}\cup\set{\neg\Prime\left(m_{j}'x+c_{j}'\right)}{j<k'}$,
for $m_{i},m_{j}'\in\Zz$ and $c_{i},c_{j}'\in N$. Now take an indiscernible
sequence $\sequence{\bar{c}_{\alpha}}{\alpha<\omega}$ starting with
$\sequence{c_{i}}{i<k}\concat\sequence{c_{j}'}{j<k'}$ over $A$.
Consider a finite union of the form $\bigcup\set{\Sigma\left(x,\bar{c}_{\alpha},\bar{c}_{\alpha}'\right)}{\alpha<l}$.
Then by indiscernibility it cannot be that $\left(m_{i},c_{i,\alpha}\right)=\left(m_{j}',c_{j,\beta}'\right)$
for some $\alpha,\beta<l$, $i<k$ and $j<k'$. Hence by (D), it is
enough to show that $\star_{\bar{f}}$ holds for $\bar{f}=\sequence{f_{i,\alpha}}{i<k,\alpha<l}$
where $f_{i,\alpha}\left(x\right)=m_{i}x+c_{i,\alpha}$, that is,
we have to show that $\varphi_{\bar{m}}\left(\sequence{\bar{c}_{\alpha}}{\alpha<l}\right)$
holds (see Remark \ref{rem:first order}). 

We have to check that if $r$ is a prime, smaller than some natural
number which depends only on $\bar{m}$, $k$ and $l$, (so in particular
a standard prime number), for some $0\leq t<r$, for all $i<k$ and
$\alpha<l$, $m_{i}t+c_{i,\alpha}\not\equiv0\modp r$. If this does
not happen for $r$, then, as $c_{i,\alpha}\equiv c_{i}\modp r$,
we get that for all $0\leq t<r$, for some $i<k$, $m_{i}t+c_{i}\equiv0\modp r$.
But this means that $\Sigma$ cannot have infinitely many solutions
by Remark \ref{rem:eliminates infinity} --- contradiction.
\end{proof}
We move to NIP. We will show that $\TheoryofPrimes$ has the independence
property IP (and thus the theory is not NIP), and even the $n$-independence
property. This shows in particular that $\TheoryofPrimes$ is unstable.
We will recall the definition in the proof of Theorem \ref{thm:IP},
but the interested reader may find more in \cite{pierrebook} (about
NIP) and \cite{Chernikov2014} (on $n$-dependence). 

We will use the following proposition. 
\begin{prop}
\label{prop:arithmatic progressions}For all $n<\omega$ and $s\subseteq n$
there is an arithmetic progression $\sequence{at+b}{t<n}$ of natural
numbers such that $at+b$ is prime iff $t\in s$. \end{prop}
\begin{proof}
As we said in the introduction, according to a private communication
with Tamar Ziegler, this follows from the proof of the Green-Tao theorem
about arithmetic progression of primes \cite{GreenTao}. 

We give a very detail-free explanation of why this should be true.
Heuristically, the primes below $N$ behave like a random set of density
$1/\log N$, so the number of $x,d\leq N$ such that $x+d$, $x+2d,\ldots,x+kd$
are all primes is $N^{2}/\left(\log N\right)^{k}$. If we skip the
$i$'th element in the sequence (i.e., we do not ask it to be prime),
then the number is $N^{2}/\left(\log N\right)^{k-1}$. Hence, we may
remove all the prime arithmetic progressions and still find some sequence
where $i$'th element is not prime. 

We will however give a proof that relies on (D). Fix $n$ and $s$.
Let $b=n!+1$. Use Lemma \ref{lem:Also composite}, with the linear
maps $x+b,2x+b,\ldots,nx+b$. By Remark \ref{rem:bounding the integer},
it is enough to check that for all primes $p\leq n$, for some $t<p$,
$kt+b\not\equiv0\modp p$ for all $1\leq k\leq n$. But $b\equiv1\modp p$
so this holds for $t=0$. \end{proof}
\begin{thm}
\label{thm:IP}(Without assuming Dickson's conjecture) $\TheoryofPrimes$
has the independence property and even the $n$-independence property.
Hence so does $\TheoryofPrimes^{*}$. \end{thm}
\begin{proof}
We use only Proposition \ref{prop:arithmatic progressions}. To say
that $T$ is $n$-independent, we have to find a formula $\varphi\left(x,y_{1},\ldots,y_{n}\right)$
such that for all $k<\omega$, there are tuples $a_{i,j}$ for $i<n,j<k$
inside some model $M\models T$ such that for every subset $s\subseteq k^{n}$,
there is some $b_{s}\in M$ with $M\models\varphi\left(b_{s},a_{0,j_{0}},\ldots a_{n-1,j_{n-1}}\right)$
iff $\left(j_{0},\ldots,j_{n-1}\right)\in s$. This of course implies
the independent property. 

The formula we take is $\varphi\left(x,y_{1},\ldots,y_{n}\right)=\Prime\left(x+y_{1}+\cdots+y_{n}\right)$,
and we work in $\Zz$.

Given $k$, by Proposition \ref{prop:arithmatic progressions} there
is an arithmetic progression of length $k^{n}\cdot2^{\left(k^{n}\right)}$,
which we write as $\sequence{\bar{c}_{s}}{s\subseteq k^{n}}$ where
$\bar{c}_{s}=\sequence{c_{s,l}}{l<k^{n}}$, such that for each subset
$s\subseteq k^{n}$ and $l<k^{n}$, $\Prime\left(c_{s,l}\right)$
iff $\left(j_{0},\ldots,j_{n-1}\right)\in s$ where $j_{i}<k$ are
(unique) such that $l=\sum_{i<n}j_{i}k^{i}$. 

Suppose this progression has difference $d>0$. Now we choose $a_{i,j}$
for $i<n,j<k$ and $b_{s}$ for $s\subseteq k^{n}$ as follows.

Let $a_{0,j}=j\cdot d$ for $j<k$ and in general, for $i<n$, $a_{i,j}=jd\cdot k^{i}$.
Let $b_{s}=c_{s,0}$. 

Now note that 
\[
c_{s,0}+\sum_{i<n}\left(j_{i}d\right)k^{i}=c_{s,\sum_{i<n}j_{i}\cdot k^{i}}.
\]

And so we are done. 
\end{proof}
\bibliographystyle{alpha}
\bibliography{common2}

\end{document}